\documentclass[sumlimits,intlimits,reqno,12pt]{amsart}
\usepackage{amsmath, amsfonts, amssymb,amsthm}
\theoremstyle{plain}
\usepackage[all]{xy}
\newtheorem{theorem}{Theorem}[section]
\newtheorem{lemma}[theorem]{Lemma}
\newtheorem{corollary}[theorem]{Corollary}
\newtheorem{proposition}[theorem]{Proposition}

\theoremstyle{definition}
\newtheorem{definition}[theorem]{Definition}

\theoremstyle{remark}

\numberwithin{equation}{section}

\def\glc{generalized local cohomology }

\def\fg{finitely generated }
\def\inf{\mathop{\mathrm{inf}}}
\def\Ext{\mathrm{Ext}}
\def\pd{\mathop{\mathrm{pd}}}
\def\Tor{\mathrm{Tor}}
\def\Supp{\mathop{\mathrm{Supp}}}

\def\Hom{\mathrm{Hom}}
\def\hom{\mathrm{Hom}}
\def\Ass{\mathrm{Ass}}\def\Supp{\mathrm{Supp}}
\def\P{\mathfrak{p}}\def\p{\mathfrak{p}}

\def\ker{\mathop{\mathrm{Ker}}}
\def\im{\mathop{\mathrm{Im}}}
\def\tot{\mathrm{Tot}}
\begin{document}

\title{SOME PROPERTIES OF GENERALIZED LOCAL COHOMOLOGY MODULES WITH RESPECT TO A PAIR OF IDEALS}


\author{TRAN TUAN NAM}
\address{Department of Mathematics-Informatics, Ho Chi Minh University of  Pedagogy, Ho Chi Minh city, Viet Nam.}
\curraddr{}
\email{namtuantran@gmail.com}
\thanks{This research is funded by Vietnam National Foundation for
Science and Technology Development (NAFOSTED)}

\author{NGUYEN MINH TRI}
\address{Department of Natural Science Education, Dong Nai University, Dong Nai, Viet Nam.}
\curraddr{}
\email{triminhng@gmail.com}
\thanks{}



\date{}

\dedicatory{}

\begin{abstract} We introduce a notion of generalized local cohomology
 modules with respect to a pair of ideals $(I,J)$ which is a generalization of the concept of local cohomology modules  with respect to $(I,J).$
 We  show that generalized local cohomology modules $H^i_{I,J}(M,N)$ can be
computed by the \v{C}ech cohomology modules. We also  study the
artinianness of generalized local cohomology modules
$H^i_{I,J}(M,N).$
\end{abstract}

\maketitle


 \markboth{TRAN TUAN NAM, NGUYEN MINH TRI}{Some properties  of generalized local cohomology modules...}
\bigskip
 \section{Introduction}
 \medskip

Throughout this paper, $R$ is a noetherian commutative (non-zero
identity) ring. In \cite{takloc}, Takahashi, Yoshino and Yoshizawa
introduced the local cohomology modules with respect to a pair of
ideals $(I,J)$. For an $R$-module $M$,
 the $(I,J)$-torsion submodule of $M$ is $\Gamma_{I,J}(M)=\{x\in M | I^nx\subset Jx \text{ for some positive integer } n\}$. $\Gamma_{I,J}$
 is a covariant  functor from the category of $R$-modules to itself.
  The $i$-th  local cohomology functor $H^i_{I,J}$ with respect to $(I,J)$ is defined to be the $i$-th right derived functor of $\Gamma_{I,J}$.
   When $J=0$, the $H^i_{I,J}$ coincides with the usual local cohomology functor $H^i_I.$

For two $R-$modules $M$ and $N,$ we  define  $\Gamma_{I,J}(M,N)$ to
be the $(I,J)$-torsion submodule of $\Hom_R(M,N).$ For each
$R-$module $M,$ there is a covariant functor $\Gamma_{I,J}(M,-)$
from the category of $R$-modules to itself. The $i$-th generalized
local cohomology functor $H^i_{I,J}(M,-)$ with the respect to pair
of ideals $(I,J)$ is the $i$-th right derived functor of
$\Gamma_{I,J}(M,-).$ This definition is really a generalization of
the local cohomology functors $H^i_{I,J}$ with respect to $(I,J).$

The organization of the paper is as follows. In the next section, we
study  some elementary properties of \glc modules with respect to a
pair of ideals $(I,J).$ We also show that generalized local
cohomology modules $H^i_{I,J}(M,N)$ can be computed by \v{C}ech
cohomology modules (Theorem \ref{cechcomplex}).

The last section is devoted to study  the artinianness of local
cohomology modules $H^i_{I,J}(M,N).$
   In Theorem \ref{Artinian1} we prove that if $M,\ N$ are two \fg $R$-modules with $p=\pd(M)$ and $d=\dim(N),$ then
$H^{r+d}_{I,J}(M,N)\cong \Ext^r_R(M,H^d_{I,J}(N))$ and
$H^{r+d}_{I,J}(M,N)$ is an artinian $R-$module. Theorem
\ref{Artinian2} shows that if $M$ is a \fg  $R$-modules and
$H^i_{I,J}(N)$ is artinian for all $i<t,$ then
 $H^i_{I,J}(M,N)$  is artinian for all $i<t.$ On the other hand,   $\Ext^{i}_R(R/\mathfrak{a}, N)$ is also artinian for all $i<t$ and for all $\mathfrak{a}\in \tilde{W}(I,J).$
Let $I,J$ be two ideals of the local ring $(R,\mathfrak{m})$ such
that $\sqrt{I+J}=\mathfrak{m}$ and  $M, N$ are two \fg $R$-modules
with $\dim(N)<\infty.$ If $H^i_{I,J}(M,N)$ is an artinian $R$-module
for all $i>t,$ then $H^t_{I,J}(M,N)/JH^t_{I,J}(M,N)$ is also an
artinian $R$-module (Theorem \ref{artinian3}). This section is
closed by Theorem \ref{artiniannisartinian} which says that
$H^i_{I,J}(M,N)$ is artinian for all $i\geq 0$ provided $M$ is a
finitely generated $R$-module and $N$ is an artinian $R$-module.
\bigskip

\section{Some basic properties of  \glc modules with respect to a pair of ideals}
\medskip
Let $I,\ J$ be two ideals of $R.$  For an $R$-module $M$,
 the $(I,J)$-torsion submodule of $M$ is $$\Gamma_{I,J}(M)=\{x\in M | I^nx\subset Jx \text{ for some positive integer } n\}
 \text{(\cite{takloc})}.$$
  We introduce the following
definition.

\begin{definition}
For two $R-$modules $M,\ N$ we denote by $\Gamma_{I,J}(M,N)$ the
following module
$$\Gamma_{I,J}(M,N)=\Gamma_{I,J}(\Hom_R(M,N)).$$
\end{definition}

In the special case $M=R,$\ $\Gamma_{I,J}(R,N)=\Gamma_{I,J}(N)$ the
$(I,J)$-torsion submodule of $N.$ Note that an element $f\in
\Gamma_{I,J}(M,N)$ if and only if there is an integer $n>0$ such
that $I^nf(x)\subset Jf(x)$
 for all $x\in M$.

For each  $R$-module $M$,  $\Gamma_{I,J}(M,-)$ is a  left exact
covariant functor from the category of $R$-modules to itself.

Let us denote by $H^i_{I,J}(M,-)$ the $i$-th right derived functor
of $\Gamma_{I,J}(M,-)$ and call the $i$-th generalized local
cohomology functor with the respect to pair of ideals $(I,J)$.

\begin{theorem}\label{defgijhom} Let $M$ be a finitely generated
$R$-module and $N$ an $R-$module. Then
$$\Gamma_{I,J}(M,N)= \Hom_R(M,\Gamma_{I,J}(N)).$$
\end{theorem}
\begin{proof}
If $f\in \Gamma_{I,J}(M,N)$,  there exists an integer $n>0$ such
that $I^nf(x)\subset Jf(x)$ for all $x\in M$. Since $f(x)\in N$, we
get $f(x)\in \Gamma_{I,J}(N)$ for all $x\in M$ and then $f\in
\Hom_R(M,\Gamma_{I,J}(N)).$

Let $f\in \Hom_R(M,\Gamma_{I,J}(N))$. Assume that $x_1, x_2,\ldots, x_m$ are generators of $M$.

Since $f(x_i)\in \Gamma_{I,J}(N)$, there exist an integer $n_i$ such that $I^{n_i}f(x_i)\subset Jf(x_i)$ for $i=1,2,\ldots, m.$

Set $n=n_1n_2\ldots n_m$, then $I^nf(x_i)\subset Jf(x_i)$ for all $i=1,2,\ldots,m.$

It follows $I^nf(x)\subset Jf(x)$ for all $x\in M$. So $I^nf\subset Jf$ and then $f\in \Gamma_{I,J}(\Hom_R(M,N))=\Gamma_{I,J}(M,N).$
\end{proof}

Note that in \cite{zamgen} Zamani introduced an other definition of
local cohomology functors $H^i_{I,J}$ as follow
$$H^i_{I,J}(M,N)=H^i(\Hom_R(M,\Gamma_{I,J}(E^\bullet)))$$for all
$i\geq 0,$ where $E^\bullet$ is an injective resolution of
$R$-module $N.$ Thus from \ref{defgijhom}  we see that our
definition is coincident with Zamani's one.

  We have a property of the set of associated
primes  of $\Gamma_{I,J}(M,N)$.
\begin{corollary}\label{AssGIJ}
Let $M$ be a finitely generated $R$-module and $N$ an $R-$module.
Then
$$\Ass(\Gamma_{I,J}(M,N))=\Supp(M)\cap \Ass(N)\cap W(I,J).$$
\end{corollary}
\begin{proof}
Since $M$ is a finitely generated $R$-module,
$\Ass(\Hom_R(M,K))=\Supp(M)\cap \Ass(K)$ for all $R$-modules $K.$ By
\ref{defgijhom}, we have
\begin{align*}
\Ass(\Gamma_{I,J}(M,N))&=\Ass(\Gamma_{I,J}(\Hom_R(M,N)))\\
&=\Ass(\Hom_R(M,\Gamma_{I,J}(N)))\\
&=\Supp(M)\cap \Ass(\Gamma_{I,J}(N))\\
&=\Supp(M)\cap \Ass(N)\cap W(I,J)
\end{align*}
as required.
\end{proof}

The following proposition  is an extension of \cite[1.4]{takloc}.

\begin{proposition}\label{progij}
Let $M$ be a finitely generated $R$-module and $N$ an $R-$module.
Let $I, I^{'}, J, J^{'}$ be ideals of $R$. Then
\begin{enumerate}
\item[(i)] $\Gamma_{I,J}(\Gamma_{I^{'},J^{'}}(M,N))=\Gamma_{I^{'},J^{'}}(\Gamma_{I,J}(M,N)).$
\item[(ii)] If $I\subseteq I',$ then $\Gamma_{I,J}(M,N)\supseteq \Gamma_{I',J}(M,N).$
\item[(iii)] If $J\subseteq J',$ then $\Gamma_{I,J}(M,N)\subseteq \Gamma_{I,J'}(M,N).$
\item[(iv)] $\Gamma_{I,J}(\Gamma_{I^{'},J}(M,N))=\Gamma_{I+I^{'},J}(M,N).$
\item[(v)] $\Gamma_{I,J}(\Gamma_{I,J^{'}}(M,N))=\Gamma_{I,JJ'}(M,N)=\Gamma_{I,J\cap J^{'}}(M,N).$ Moreover, $H^i_{I,JJ'}(M,N)=H^i_{I,J\cap J'}(M,N)$
for all  $i.$
\item[(vi)] If $J^{'}\subseteq J$, then $\Gamma_{I+J^{'},J}(M,N)=\Gamma_{I,J}(M,N).$ Moreover,

$\Gamma_{I+J,J}(M,N)=\Gamma_{I,J}(M,N)$ and
$H^i_{I+J,J}(M,N)=H^i_{I,J}(M,N)$ for all   $i.$
\item[(vii)] If $\sqrt{I}=\sqrt{I^{'}}$, then $H^i_{I,J}(M,N)=H^i_{I^{'},J}(M,N)$ for all  $i.$ In particular, $H^i_{I,J}(M,N)=H^i_{\sqrt{I},J}(M,N)$ for all  $i.$
\item[(viii)] If $\sqrt{J}=\sqrt{J^{'}}$, then $H^i_{I,J}(M,N)=H^i_{I,J^{'}}(M,N)$ for all  $i$.
\end{enumerate}
\end{proposition}
\begin{proof}
We only prove (i), the others are similar.

Combining \cite[1.4]{takloc} and \ref{defgijhom}, we have
\begin{align*}
\Gamma_{I,J}(\Gamma_{I^{'},J^{'}}(M,N))&=\Gamma_{I,J}(\Hom_R(M,\Gamma_{I'.J'}(N)))\\
&=\Hom_R(M,\Gamma_{I,J}(\Gamma_{I',J'}(N)))\\
&=\Hom_R(M,\Gamma_{I',J'}(\Gamma_{I,J}(N)))\\
&=\Gamma_{I^{'},J^{'}}(\Gamma_{I,J}(M,N))
\end{align*}
as required.
\end{proof}

\begin{lemma}\label{G_IJinject}
If $E$ is an injective $R$-module, then $\Gamma_{I,J}(E)$ is also
injective.
\end{lemma}
\begin{proof}
From \cite[3.2]{takloc} we have
$$\Gamma_{I,J}(E)\cong  \mathop {\lim }\limits_{\begin{subarray}{c}
   \longrightarrow  \\
   \mathfrak{a}\in \tilde{W}(I,J)
\end{subarray}} \Gamma_\mathfrak{a}(E),$$ where $\tilde{W}(I,J)$ is the
set of ideals $\mathfrak{a}$ of $R$ such that $I^n\subset
\mathfrak{a}+J$ for some integer $n.$

 Since $E$ is an injective $R$-module, $\Gamma_\mathfrak{a}(E)$
is also injective. Moreover, $R$ is a Noetherian ring,
 then the direct limit of injective modules is injective. Therefore we have  the conclusion.\end{proof}

It is well-known that $ H^i_{I}(M,N)\cong\Ext^i_R(M,N),$ where $N$ is an $I$-torsion $R$-module.
The following proposition gives a similar result when $N$ is $(I,J)$-torsion.

\begin{proposition}\label{H_IJExt}
Let $N$ be an $(I,J)$-torsion $R$-module. Then $$H^i_{I,J}(M,N)\cong
\Ext^i_R(M,N)$$ for all $i\geq 0.$
\end{proposition}
\begin{proof}
From \cite[1.12]{takloc} there exists an injective resolution
$E^\bullet$ of $N$ such that each term is an $(I,J)$-torsion
$R$-module. Then we have by \ref{defgijhom}
\begin{align*}
H^i_{I,J}(M,N)&\cong H^i(\Hom_R(M,\Gamma_{I,J}(E^\bullet)))\\
&\cong H^i(\Hom_R(M,E^\bullet))\\
&=\Ext^i_R(M,N)
\end{align*}
 for all $i\geq 0.$
\end{proof}

When $N$ is a $J$-torsion $R$-module,  we have  the following
proposition.
\begin{proposition}\label{Nisjtorsion}
If $N$ is a $J$-torsion $R$-module, then $$H^i_{I,J}(M,N)\cong
H^i_I(M,N)$$ for all $i\geq 0.$
\end{proposition}
\begin{proof}  It is obvious that $\Gamma_I(N)\subset \Gamma_{I,J}(N).$
Let $x\in \Gamma_{I,J}(N)$, there exist integers $n,k$  such that
$I^nx\subset Jx$ and $J^kx=0$. Hence $I^{nk}x=0$ and then $x\in
\Gamma_I(N)$. Thus $\Gamma_{I,J}(N)=\Gamma_I(N).$

It remains to prove that $\Gamma_{I,J}(M,N)\cong \Gamma_I(M,N)$.
From \ref{defgijhom} we have
\begin{align*}
\Gamma_{I,J}(M,N)&= \Hom_R(M,\Gamma_{I,J}(N))\\
&=\Hom_{R}(M,\Gamma_I(N))\\
&\cong \Gamma_I(M,N).
\end{align*}
By the property of derived functors, we obtain $H^i_{I,J}(M,N)\cong
H^i_I(M,N)$ for all  $i\geq 0.$
\end{proof}

Let  $J$ be an ideal of $R.$ For an element $a\in R$  the set
$$S_{a,J}=\{a^n+j\mid n\in\mathbb{N}, j\in J\}$$
is a multiplicatively closed subset of $R$ (\cite[2.1]{takloc}). Let
$M$ be a finitely generated $R$-module. Denote by $M_{a,J}$ the
module of fractions of  the $R-$module $M$ with respect to
$S_{a,J}.$ The complex $C^\bullet_{a,J}$ was  given by
$$C^\bullet_{a,J}: 0\to R\to R_{a,J}\to 0.$$

For a sequence $\boldsymbol{a} =\{a_1,a_2,\ldots,a_r\}$ of elements
of $R$, the \v{C}ech complex $C^\bullet_{\boldsymbol{a},J}$ was
defined as
\begin{footnotesize}
\begin{align*}
C^\bullet_{\boldsymbol{a},J}&=\bigotimes_{i=1}^r C^\bullet_{a_i,J}\\
&=\Big( 0\to R\to \prod_{i=1}^r R_{a_i,J}\to \prod_{i<j}(R_{a_i,J})_{a_j,J}\to\cdots\to(\cdots (R_{a_1,J})\cdots)_{a_r,J}\to 0 \Big).
\end{align*}
\end{footnotesize}

In \cite[2.4]{takloc}, there is a natural isomorphism
$H^i_{I,J}(M)\cong H^i(C^\bullet_{\boldsymbol{a},J}\otimes_R M),$
where $\boldsymbol{a}=\{a_1,a_2,\ldots,a_r\}$ is a sequence of
elements of $R$ that generates $I.$

Let $$\mathbf{F}_\bullet: \cdots\longrightarrow F_2\longrightarrow
F_1\longrightarrow F_0\longrightarrow M \longrightarrow 0$$be a free
resolution of $M$ with the finitely generated free modules.

Apply the functor $\Hom_R(-,N)$ to above resolution, we have a
complex
$$\Hom_R(\mathbf{F_\bullet},N): 0 \rightarrow \Hom_R(M,N)\rightarrow \Hom_R(F_0,N) \rightarrow \Hom_R(F_1,N) \rightarrow
....$$
 Then there is a bicomplex
$C^\bullet_{\boldsymbol{a},J}\otimes_R
\Hom_R(\mathbf{F_\bullet},N)=\{C^p_{\boldsymbol{a},J}\otimes_R
\Hom_R(F_q,N)\},$ where $C^p_{\boldsymbol{a},J}$ is the $p$-th
position in the \v{C}ech complex $C^\bullet_{\boldsymbol{a},J}.$
Thus we get a total complex $\tot(M,N)$ of bicomplex
$C^\bullet_{\boldsymbol{a},J}\otimes_R \Hom_R(\mathbf{F_\bullet},N)$
where
$$\tot(M,N)^n=\bigoplus_{p+q=n}C^p_{\boldsymbol{a},J}\otimes_R
\Hom_R(F_q,N).$$

\begin{theorem}\label{cechcomplex}
Let $M$ be a finitely generated $R$-module. Then for all $R$-modules $N$ and $n\geq 0,$
$$H^n_{I,J}(M,N)\cong H^n(\tot(M,N)).$$
\end{theorem}
\begin{proof}
It is clear that $\{H^n(\tot(M,-))\}_n$ and $\{H^n_{I,J}(M,-)\}_n$
are exact connected right sequences of functors.

We next prove that $H^0(\tot(M,N))\cong H^0_{I,J}(M,N).$ Consider
the homomorphism  $d^0: \tot(M,N)^0\to \tot(M,N)^1.$ As
$C^0_{\boldsymbol{a},J}=R,$ we get
\begin{footnotesize}
$$
H^0(\tot(M,N))  \cong\ker(\hom_R(F_0,N)\overset{d^0}\to
\hom_R(F_1,N)\oplus (C^1_{\boldsymbol{a},J}\otimes_R
\hom_R(F_0,N))).
$$\end{footnotesize}
Thus
\begin{align*}
H^0(\tot(M,N)) & = \hom_R(M,N)\bigcap \Gamma_{I,J}(\hom_R(F_0,N))\\
&= \Gamma_{I,J}(\hom_R(M,N))=H^0_{I,J}(M,N)
\end{align*}
by \cite[2.3(5)]{takloc}.

The proof is completed by  showing that $H^n(\tot(M,E))=0$ for all
$n>0$ and for any injective $R$-module $E.$ It follows from
\cite[10.18]{rotani} a spectral sequence
$$E^{p,q}_2= H''^pH'^q(C^\bullet_{\boldsymbol{a},J}\otimes_R
\Hom_R(\mathbf{F_\bullet},E))\underset{p}\Rightarrow
H^n(\tot(M,E)).$$ Note that
$$E_1^{p,q}=H^q(C^\bullet_{\boldsymbol{a},J}\otimes
\hom_R(F_p,E)).$$
 From the proof of \cite[2.4]{takloc}, $H^i(C^\bullet_{\boldsymbol{a},J}\otimes_R E)=0$ for all
 $i>0$ and for any injective $R-$module $E$.
 Note that $\hom_R(F_q,E)$ is also an injective $R$-module  for all $q\geq
 0.$ Hence
$$E_1^{p,q}=\begin{cases}\begin{array}{ll}
0&,q>0\\
\ker( \Hom_R(F_p,E)\to
\underset{i=1}{\overset{r}{\prod}}\Hom_R(F_p,E)_{a_i,J})&,q=0.
\end{array}
\end{cases}
$$
Combining \cite[2.3(5)]{takloc} with \ref{defgijhom} yields
\begin{align*}
\ker( \Hom_R(F_p,E)\to
\underset{i=1}{\overset{r}{\prod}}\Hom_R(F_p,E)_{a_i,J})&\cong
\Gamma_{I,J}(\hom_R(F_p,E))\\
&\cong \hom_R(F_p,\Gamma_{I,J}(E)).
\end{align*}
It follows
$$
E_2^{p,q}=\begin{cases}\begin{array}{ll} 0&,q>0\\
H^p(\hom_R(F_\bullet,\Gamma_{I,J}(E)))&,q=0.
\end{array}
\end{cases}
$$
As $\Gamma_{I,J}(E)$ is an injective $R-$module, the following
sequence is exact
$$\hom_R(F_\bullet,\Gamma_{I,J}(E)): 0 \rightarrow \Hom_R(M,\Gamma_{I,J}(E))\rightarrow \Hom_R(F_0,\Gamma_{I,J}(E)) \rightarrow$$
$$ \rightarrow\Hom_R(F_1,\Gamma_{I,J}(E)) \rightarrow ....$$
Thus $ E_2^{p,0} = 0$ for all $p>0.$
 From \cite[10.21
(ii)]{rotani} we get $$H^n(\tot(M,E))\cong {E_2^{n,0}} = 0$$ for all
$n>0.$ The proof is complete.
\end{proof}
\bigskip

\section{On artinianness of  \glc modules with respect to a pair of ideals}
\medskip

We have the following theorem.
\begin{theorem}\label{Artinian1}
Assume that $(R,\mathfrak{m})$ is a local ring. Let $M,N$ be two \fg
$R$-modules with $r=\pd(M)$ and $d=\dim(N)$. Then
$$H^{r+d}_{I,J}(M,N)\cong \Ext^r_R(M,H^d_{I,J}(N)).$$Moreover
$H^{r+d}_{I,J}(M,N)$ is an artinian $R-$module.
\end{theorem}
\begin{proof}
Let $G(-)= \Gamma_{I,J}(-)$ and $F(-)=\Hom_R(M,-)$ be functors from
category of $R$-modules to itself. Then $FG=\Gamma_{I,J}(M,-)$ and
$F$ is left exact.  For any injective module $E$
$$R^iF(G(E))=R^i\Hom_R(M, \Gamma_{I,J}(E)) =0$$ for all $i>0,$ as
$\Gamma_{I,J}(E)$ is an injective $R-$module. By
\cite[10.47]{rotani}
 there is a Grothendieck
spectral sequence
$$E_2^{pq}={\Ext}^p_R(M,H^q_{I,J}(N))\underset{p}{\Longrightarrow} H^{p+q}_{I,J}(M,N).$$
We now consider the homomorphisms of the spectral
$$E_k^{r-k,d+k-1}\to E_k^{r,d}\to E_k^{r+k,d+1-k}.$$
 We have $H^q_{I,J}(N)=0$ for all $q>d$  by  \cite[4.7]{takloc}. Then  $E_2^{pq}=0$ for all $p>r$
 or $q>d.$ Thus
$E_k^{r-k,d+k-1}=E_k^{r+k,d+1-k}=0$ for all $k\geq2,$ so
$$E_2^{r,d} = E_3^{r,d} = ... = E_{\infty}^{r,d}.$$
It remains to prove that $E_{\infty}^{r,d}\cong H_{I,J}^{r+d}(M,N).$
Indeed, there is a filtration $\Phi$ of $H^{r+d}=H_{I,J}^{r+d}(M,N)$
such that
$$0 = \Phi^{r+d+1}H^{r+d}\subseteq  \Phi^{r+d}H^{r+d}\subseteq
\ldots \subseteq  \Phi^{1}H^{r+d} \subseteq
\Phi^{0}H^{r+d}=H_{I,J}^{r+d}(M,N)$$ and
$$E_{\infty}^{i,r+d-i}=\Phi^iH^{r+d}/\Phi^{i+1}H^{r+d},\ 0\leq
i\leq r+d.$$ From the above proof  we have $E_2^{i,r+d-i} =
\Ext^i_R(M,H_{I,J}^{r+d-i}(N))=0$ for all $i\not= r.$ Hence
$$\Phi^{r+1}H^{r+d} = \Phi^{r+2}H^{r+d} = \ldots =  \Phi^{r+d+1}H^{r+d} = 0$$ and
$$\Phi^{r}H^{r+d} = \Phi^{r-1}H^{r+d} = \ldots = \Phi^{0}H^{r+d} =
H^{r+d}_{I,J}(M,N).$$ This gives $$E_{\infty}^{r,d} \cong \Phi^r
H^{r+d}/\Phi^{r+1}H^{r+d}\cong H^{r+d}_{I,J}(M,N).$$Thus
$\Ext^r_R(M,H^d_{I,J}(N))\cong H^{r+d}_{I,J}(M,N).$ It follows from
\cite[2.1]{chusom} that $H^d_{I,J}(N)$ is an artinian $R-$module.
Therefore $H^{r+d}_{I,J}(M,N)$ is also  an artinian $R-$module.
\end{proof}
Next theorem, we show the connection between the artinianness of $H^i_{I,J}(N)$ and $H^i_{I,J}(M,N)$.
\begin{theorem}\label{Artinian2}
Let $M$ be a \fg  $R$-modules and $N$ an $R-$module. Let  $t$ be a
positive integer. If $H^i_{I,J}(N)$ is artinian for all $i<t,$ then
\begin{enumerate}
\item[(i)] $H^i_{I,J}(M,N)$  is artinian for all $i<t.$
\item[(ii)] $\Ext^{i}_R(R/\mathfrak{a}, N)$ is artinian for all $i<t$ and for all $\mathfrak{a}\in \tilde{W}(I,J).$
\end{enumerate}
\end{theorem}
\begin{proof}
(i) We use induction on $t$. When $t=1,$ by \ref{defgijhom} we have
$$\Gamma_{I,J}(M,N)= \Hom_R(M,\Gamma_{I,J}(N)).$$ Since
$\Gamma_{I,J}(N)$ is artinian, the statement is true in this case.

Let $t>1$ and we assume that the statement is true for $t-1$ and for
any $R$-module $N$. Denote by $E(N)$ the injective hull of $N$.
Applying the functors $\Gamma_{I,J}(-)$ and $\Gamma_{I,J}(M,-)$ to
the following short exact sequence
$$0\to N\to E(N)\to E(N)/N\to 0$$
 we get isomorphisms
$$H^i_{I,J}(E(N)/N)\cong H^{i+1}_{I,J}(N)$$
and
$$
H^i_{I,J}(M,E(N)/N)\cong H^{i+1}_{I,J}(M,N)
$$
for all $i>0.$ From the hypothesis, $H^{i}_{I,J}(N)$ is artinian for
all $i<t.$  It follows that $H^i_{I,J}(E(N)/N)$ is also artinian for
all $ i< t-1$.
 By the inductive hypothesis on $E(N)/N$,  $H^i_{I,J}(M,E(N)/N)$ is artinian for all $i< t-1$.
 We conclude from the second isomorphism that  $H^i_{I,J}(M,N)$ is artinian for all $i<t.$

(ii) The proof is by induction on $t$. When $t=1,$  the short exact
sequence
$$0\to \Gamma_\mathfrak{a}(N)\to N\to N/\Gamma_\mathfrak{a}(N)\to 0.$$
deduces an exact sequence {\small $$0\to
\Hom_R(R/\mathfrak{a},\Gamma_\mathfrak{a}(N))\to
\Hom_R(R/\mathfrak{a},N)\to
\Hom_R(R/\mathfrak{a},N/\Gamma_\mathfrak{a}(N))\to\cdots$$} As
$N/\Gamma_\mathfrak{a}(N)$ is $\mathfrak{a}$-torsion-free, we have
$\Hom_R(R/\mathfrak{a},N/\Gamma_\mathfrak{a}(N))=0$ and then
$\Hom_R(R/\mathfrak{a},N)\cong
\Hom_R(R/\mathfrak{a},\Gamma_\mathfrak{a}(N))$. Note that
$\Gamma_\mathfrak{a}(N)\subset \Gamma_{I,J}(N).$ By the hypothesis,
 $\Gamma_\mathfrak{a}(N)$ is an artinian $R$-module
and then $\Hom_R(R/\mathfrak{a},\Gamma_\mathfrak{a}(N))$ is also an
artinian $R$-module.

The proof for $t>1$ is similar to (i).
\end{proof}
In \cite[2.4]{chusom}, when $(R,\mathfrak{m})$ is a local ring and
$N$ is a \fg $R$-module, there is an equality
$$\inf\{i\mid H^i_{I,J}(N)\text{ is not artinian}\}=\inf\{ \mathrm{depth} N_\p \mid \P\in W(I,J)\setminus\{\mathfrak{m}\}\}.$$
We have the following consequence.

\begin{corollary} Let $(R,\mathfrak{m})$ be a local ring. If $M$ and $N$ are two \fg $R$-modules, then
$$\inf\{\mathrm{depth} N_\P\mid \P\in W(I,J)\setminus\{\mathfrak{m}\}\} \leq \inf\{i \mid H^i_{I,J}(M,N)\text{ is not artinian}\}.$$
\end{corollary}
\begin{proof}
From \ref{Artinian2}, We have the following inequality
$$\inf\{i \mid H^i_{I,J}(N)\text{ is not artinian}\}\leq \inf\{i \mid H^i_{I,J}(M,N)\text{ is not artinian}\}.$$
Thus the conclusion follows from \cite[2.4]{chusom}.
\end{proof}

\begin{theorem}\label{artinian3}
Let $I,J$ be two ideals of the local ring $(R,\mathfrak{m})$ such
that $\sqrt{I+J}=\mathfrak{m}.$ Assume that $M, N$ are two \fg
$R$-modules with $\dim(N)<\infty$ and $t$ is an non-negative
integer. If $H^i_{I,J}(M,N)$ is an artinian $R$-module for all
$i>t,$ then $H^t_{I,J}(M,N)/JH^t_{I,J}(M,N)$ is also an artinian
$R$-module.
\end{theorem}
\begin{proof}
Combining \ref{progij}(vi) with \ref{progij}(vii) we conclude that
$H^i_{I,J}(M,N)=H^i_{\mathfrak{m},J}(M,N)$ for all $i\geq 0,$ as
$\sqrt{I+J}=\mathfrak{m}.$ Thus, without loss  of generality we can
assume that $I=\mathfrak{m}.$

We now use induction on $\dim (N)=d.$ When $d=0,$ $N$ is
$\mathfrak{m}$-torsion and then  $N$ is $(\mathfrak{m},J)$-torsion.
From \ref{H_IJExt}, there is an isomorphism
$H^i_{\mathfrak{m},J}(M,N)\cong \Ext^i_R(M,N)$ for all $i\geq 0.$
Since $N$ is artinian, it follows that $H^i_{\mathfrak{m},J}(M,N)$
is an artinian $R$-module for all $i\geq 0.$ Therefore the statement
is true in this case.

Let $d>0.$ The short exact sequence
$$0\to \Gamma_J(N)\to N\to N/\Gamma_J(N)\to 0.$$
induces a long exact sequence
$$
H^t_{\mathfrak{m},J}(M,\Gamma_J(N))\overset{\alpha}\to H^t_{\mathfrak{m},J}(M,N)\overset{\beta}\to H^t_{\mathfrak{m},J}(M,N/\Gamma_J(N))\overset{\gamma}\to\cdots
$$
Since $\Gamma_J(N)$ is a $J$-torsion $R$-module,  there is an
isomorphism $$H^i_{\mathfrak{m},J}(M,\Gamma_J(N))\cong
H^i_\mathfrak{m}(M,\Gamma_J(N))$$ by \ref{Nisjtorsion}. From
\cite[2.2]{divonv}   $H^i_{\mathfrak{m},J}(M,\Gamma_J(N))$ is
artinian for all $i.$

From the long exact sequence, we get two short exact sequences
$$
0\to \im\alpha\to H^t_{\mathfrak{m},J}(M,N)\to \im \beta \to 0
$$
and
$$
0\to \im\beta \to
H^t_{\mathfrak{m},J}(M,N/\Gamma_J(N))\to\im\gamma\to 0.
$$
Two above exact sequences deduce long exact sequences
$$
\cdots\im\alpha/J\im\alpha\to
H^t_{\mathfrak{m},J}(M,N)/JH^t_{\mathfrak{m},J}(M,N) \to
\im\beta/J\im\beta\to 0
$$
and
$$
\cdots\to \Tor^R_1(R/J,\im\gamma)\to \im\beta/J\im\beta\rightarrow
$$
$$
\to
H^t_{\mathfrak{m},J}(M,N/\Gamma_J(N))/JH^t_{\mathfrak{m},J}(M,N/\Gamma_J(N))\to
\im\gamma/J\im\gamma\to 0.
$$
Note that $\im\alpha$ and $\im\gamma$ are artinian $R$-modules. The
proof is completed by showing that
$H^t_{\mathfrak{m},J}(M,N/\Gamma_J(N))/JH^t_{\mathfrak{m},J}(M,N/\Gamma_J(N))$
is an artinian $R-$module.

Let $\overline{N}=N/\Gamma_J(N)$, then $\overline{N}$ is
$J$-torsion-free. Thus there exists an element $x\in J$ that is a
non-zerodivisor on $\overline{N}.$

The short exact sequence
$$0\to \overline{N}\overset{.x}\to \overline{N}\to \overline{N}/x\overline{N}\to 0$$
implies the following long exact sequence
$$
\cdots \to H^t_{\mathfrak{m},J}(M,\overline{N})\overset{f}\to
H^t_{\mathfrak{m},J}(M,\overline{N}/x\overline{N})\overset{g}\to
H^{t+1}_{\mathfrak{m},J}(M,\overline{N})\to\cdots
$$
From the hypothesis, we get that
$H^i_{\mathfrak{m},J}(M,\overline{N}/x\overline{N})$ is artinian for
all $i>t.$ As $\dim(\overline{N}/x\overline{N})\leq d-1$,
$H^t_{\mathfrak{m},J}(M,\overline{N}/x\overline{N})/JH^t_{\mathfrak{m},J}(M,\overline{N}/x\overline{N})$
is artinian by the inductive hypothesis.

We now consider two exact sequences
$$0\to \im f \to H^t_{\mathfrak{m},J}(M,\overline{N}/x\overline{N})\to \im g\to 0$$
and
$$H^t_{\mathfrak{m},J}(M,\overline{N})\overset{.x}\to H^t_{\mathfrak{m},J}(M,\overline{N})\to \im f\to 0.$$
They gives two long exact sequences
$$\Tor_1^R(R/J, \im g)\to \im f/ J \im f\to$$
$$
\to H^t_{\mathfrak{m},J}(M,\overline{N}/x\overline{N})/ J
H^t_{\mathfrak{m},J}(M,\overline{N}/x\overline{N}) \to \im g/ J \im
g\to 0
$$
and
$$H^t_{\mathfrak{m},J}(M,\overline{N})/JH^t_{\mathfrak{m},J}(M,\overline{N})\overset{.x}\to H^t_{\mathfrak{m},J}(M,\overline{N})/JH^t_{\mathfrak{m},J}(M,\overline{N})\to$$
$$
\to \im f/J \im f\to 0.
$$
Since $x\in J,$  we obtain from the exact sequence that
$$H^t_{\mathfrak{m},J}(M,\overline{N})/JH^t_{\mathfrak{m},J}(M,\overline{N})\cong \im f/J
\im f.$$

On other hand, $\Tor_1^R(R/J,\im g)$ is  artinian, as  $\im g\subset
H^{t+1}_{\mathfrak{m},J}(M,\overline{N})$ an  artinian $R-$module.
Hence $ \im f/ J \im f$ is an artinian $R-$module and the proof is
complete.
\end{proof}

\begin{theorem}\label{artinian6}
Let $M,N$ be two \fg $R$-modules and $t$ a positive integer such
that $H^t_{I,J}(M,R/\P)$ is artinian for all $\P\in \Supp(N)$. Then
$H^t_{I,J}(M,N)$ is also artinian.
\end{theorem}
\begin{proof}

As $N$ is finitely generated, there is a chain of submodules of $N$
$$0=N_0\subset N_1\subset N_2\subset\ldots\subset N_k=N$$
such that $N_i/N_{i-1}\cong R/\P_i$ for some $\P_i\in \Supp(N).$

For each $1\leq i\leq k$, the short exact sequence
$$0\to N_{i-1}\to N_i\to R/\P_i\to 0$$
deduces a long exact sequence
$$
\cdots\to H^t_{I,J}(M,N_{i-1})\to H^t_{I,J}(M,N_{i})\to H^t_{I,J}(M,R/\P_i)\to\cdots
$$
In particular, $H^t_{I,J}(M,N_1)\cong H^t_{I,J}(M,R/\P_1).$ From the
exact sequence, it follows that $H^t_{I,J}(M,N_i)$ is artinian for
all $1\leq i\leq k.$ This finishes the proof.
\end{proof}

From Theorem \ref{artinian6} we have the following immediate
consequence.

\begin{corollary}\label{coroSupp}
Let $M,N$ be finitely generated $R$-modules and $t$ a positive integer. Assume that $H^t_{I,J}(M,R/\P)$ is artinian for all $\P\in \Supp(N).$

(i) If $L$ is a finitely generated $R$-module such that
$\Supp(L)\subset \Supp(N),$  then $H^t_{I,J}(M,L)$ is artinian.

(ii) If $\mathfrak{a}$ is an ideal of $R$ such that
$V(\mathfrak{a})\subset \Supp(N),$ then
$H^t_{I,J}(M,R/\mathfrak{a})$ is artinian.
\end{corollary}

 In the following theorem, we study the artinianness of
$H^i_{I,J}(M,N)$ when $N$ is artinian.

\begin{theorem}\label{artiniannisartinian}
Let $M$ be a finitely generated $R$-module and $N$ an artinian
$R$-module. Then $H^i_{I,J}(M,N)$ is artinian for all $i\geq 0.$
\end{theorem}
\begin{proof}
We use induction on $i$. When $i=0,$ we have
$\Gamma_{I,J}(M,N)=\Hom_R(M,\Gamma_{I,J}(N))$ is artinian, as
$\Gamma_{I,J}(N)\subset N.$

Let $i>0,$ denote by $E(N)$  the injective hull of $N$. Note that,
if $N \subset K$ is an essential submodule, then $N$ is artinian if
and only if $K$ is artinian. Hence $E(N)$ is also artinian.

Now the short exact sequence
$$0\to N\to E(N)\to E(N)/N\to 0$$
deduces a long exact sequence
$$\cdots\to H^{i-1}_{I,J}(M,E(N)/N)\to H^i_{I,J}(M,N)\to H^i_{I,J}(M,E(N))\to\cdots$$
Since $H^i_{I,J}(M,E(N))=0$ for all $i>0$, there are isomorphims
$$H^{i-1}_{I,J}(M,E(N)/N)\cong H^i_{I,J}(M,N)$$
for all $i>1.$

$H^{i-1}_{I,J}(M,E(N)/N)$ is artinian by inductive hypothesis.
Therefore $H^i_{I,J}(M,N)$ is also artinian.
\end{proof}


\bigskip

\begin{thebibliography}{90}
\bibitem{broloc}  M. P. Brodmann, R. Y. Sharp, \emph{Local cohomology: an algebraic introduction with geometric applications}, Cambridge University Press, 1998.

\bibitem{chusom} L. Chu, Q. Wang, "Some results on local cohomology modules defined by a pair of ideals", {\it J. Math. Kyoto Univ}, 49:193-200, 2009.

\bibitem{cuoont} N. T. Cuong, N. V. Hoang, "On the vanishing and the finiteness of
supports of generalized local cohomological modules," {\it
Manuscripta Math.} 126(2008), 59-72.

\bibitem{divonv} Divaani-Aazar, Sazeedeh, Tousi, "On the vanishing of \glc modules", \textit{Algebra Colloquium}, Vol. 12, No. 2 (2005) 213-218.

\bibitem{hasonv} S. H. Hassanzadeh, A. Vahidi, "On vanishing and cofiniteness of \glc modules", {\it Communication of Algebra}, 37: 2290-2299,2009

\bibitem{melona} L. Melkersson, "On asymptotic stability for sets
of prime ideals connected with the powers of an ideal,"  {\it Math.
Proc. Camb. Phil. Soc.,} 107 (1990), 267-271.

\bibitem{namont1} T. T. Nam, "On the non-vanishing and the artinianness of  generalized local cohomology modules", {\it Journal of Algebra and Its Applications} (to appear).

\bibitem{rotani} J. Rotman, \emph{An introduction to homological algebra, 2nd edition}, Springer, 2009.

\bibitem{suzont} N. Suzuki, "On the generalized local cohomology
and its duality," {\it J. Math. Kyoto Univ. (JAKYAZ),} 18-1 (1978),
71-85.

\bibitem{takloc} R. Takahashi, Y. Yoshino, T. Yoshizawa, "Local cohomology based on a nonclosed support defined by a pair of ideals", {\it J. Pure Appl. Algebra}, 213: 582-600, 2009.

\bibitem{zamgen} N. Zamani, "Generalized local cohomology relative to ($I,J$)", {\it Southeast Asian Bullettin of Mathematics}, 35: 1045-1050, 2011.
\end{thebibliography}
\end{document}